\documentclass[12pt]{amsart}
\usepackage{amssymb}
\usepackage{enumitem}

\theoremstyle{plain}
\newtheorem{theorem}{Theorem}[section]

\theoremstyle{definition}
\newtheorem{definition}[theorem]{Definition}
\newtheorem{example}[theorem]{Example}

\newtheorem{question}[theorem]{Question}
\theoremstyle{remark}
\newtheorem{remark}[theorem]{Remark}
\newcommand{\Aut}{\operatorname{Aut}}

\newcommand{\G}{\mathcal{G}}

\newcommand{\Cc}{{\frak G}}
\newcommand{\F}{{\mathcal F}}

\title[Faithful actions on DGA and the group isomorphism problem]{Faithful actions on Differential Graded Algebras determine the isomorphism type of a large class of groups}
\author{Cristina Costoya \and Antonio Viruel
}

\thanks{First author is partially  supported by MICINN grant
MTM 2009-14464-C02 (European FEDER support included), and by Xunta de Galicia grant Incite 09 207 215
PR. Second author is partially supported by MICINN grant MTM2010-18089 (European FEDER support included), and JA grants FQM-213 and P07-FQM-2863}
\address{\small \rm  Dep.\ de Computaci\'on, \'Alxebra,
Universidade da Coru{\~n}a, Campus de Elvi{\~n}a, s/n, 15071 A Coru{\~n}a, Spain.}
\email{cristina.costoya@udc.es}
\address{\small \rm  Dep.\ de \'Algebra, Geometr{\'\i}a y Topolog{\'\i}a,
Campus de Teatinos, Universidad de M\'alaga,
29071 M\'alaga, Spain}
\email{viruel@agt.cie.uma.es}

\begin{document}

\begin{abstract}
We prove that  the isomorphism type of a large class of groups (containing finite groups, countable Artinian groups and mapping class groups of certain surfaces, among others) is determined by the set of differential graded $\mathbb Q$-algebras on which these groups act faithfully. 
\end{abstract}

\maketitle

\section{Introduction}
In a category ${\mathcal C}$,  the Isomorphism Problem consists on providing a procedure that determines whether two objects in ${\mathcal C}$ are isomorphic or not.  If such a procedure exists and the techniques fit into some theory $\mathcal T$, it is said that
 {$\mathcal T$ tells objects in ${\mathcal C}$ apart}.
Well known examples of this problem are the Dehn's Finite Presentation Problem \cite{Dehn} or
the Graph Isomorphism Problem \cite{RC, Gati}. 
In this paper we consider
the Group Isomorphism Problem and the following question.
\begin{question}\label{q2} 
Does {Representation Theory} tell groups
apart?
\end{question}
The answer clearly depends on what is
meant by {Representation Theory} and how it is used to compare groups.
If  {Representation Theory}  is  just considered as
linear representation of groups, where two groups are compared
by looking at their set of modules, then  Question
\ref{q2} is settled in the negative.  In fact, in a remarkable paper of Hertweck \cite{Hertweck},  
two non isomorphic finite groups $G$ and $H$ are constructed, both of size
$2^{21}97^{28}$,  such that ${\mathbb Z}[G]\cong {\mathbb Z}[H]$ as
rings. Therefore, the isomorphism type of a  finite group cannot be decided by enumerating the set of vector spaces on which that group acts faithfully.

In this paper, we approach Representation Theory  in a broader way
by considering actions on
differential graded algebras over $\mathbb Q$, which will be abbreviated as DGA's.  In that context and for an appropriate class $\Cc$  of groups (containing finite groups, countable Artinian groups, mapping class groups of certain surfaces ...) we settle Question \ref{q2}  in the positive.  For convenience of the reader, the class $\Cc$  will be introduced in the next section (see Definition \ref{def:C}). The main result in this paper is the following.
\begin{theorem}\label{main} 
For $G$ and $H$ groups in  $\Cc$, the following statements are equivalent:
\begin{itemize}
\item[i)] $G$ and $H$ are isomorphic.
\item[ii)] For every { \rm DGA}, $(A,d)$, $G\leq\operatorname{Aut}(A,d)$ if and
only if  $H\leq\operatorname{Aut}(A,d)$, where $\operatorname{Aut}(A,d)$ denotes the automorphism group of $(A,d)$.
\end{itemize}%
\end{theorem}
Next section provides the necessary tools for the construction of  the class $\Cc$ in Theorem \ref{main} (see Definition \ref{def:C}) and also, a list of examples to illustrate it. The proof of this theorem will be postponed to Section \ref{sec:main}. 
\section{The class $\Cc$ of groups}\label{sec:C}

The first step towards the understanding of $\Cc$ is to think of groups as the automorphisms group of graphs.  Following \cite{Hell-N}, a simple and undirected graph $\G$ is a pair $(V,E)$ where $V$ is the set of vertices of $\G$ and $E$ is a set of $2$-elements subsets of $V$, called the set of edges of $\G$. Given $v\in V$,  the degree of $v$ is defined by $d(v,\G)=|\{w : [v,w]\in E\}|$. The graph $\G$ is said locally finite if $d(v,\G)<\infty$ for every $v\in V$.  Now, given two graphs $\G_i=(V_i,E_i)$, $i=1,2$ and $f\colon V_1\to V_2$ a mapping,  $f$ is called a homomorphism of $\G_1$ into $\G_2$ if $[f(v),f(w)]\in E_2$ whenever $[v,w]\in E_1$; the automorphisms are defined in the usual way.

In general, any group can be represented as group of automorphisms of a graph \cite{Sabidusi}, that is, for every group $G$ there exists a graph $\G$ such that $\Aut(\G)\cong G$. Unfortunately, if restrictions on the finitude of the degree of vertices are required, there exist groups that cannot be represented as groups of automorphisms of  locally finite graphs \cite[p.\ 250]{Hell-N}.  In that sense, the best possible result  is the following.

\begin{theorem}\cite[Theorem 4]{Hell-N}\label{th:groups&graphs}
Let $G$ be a group of countable order, then there exists a connected locally finite graph $\G$ such that $\Aut(\G)\cong G$. If $G$ is finite, $\G$ can also be chosen to be finite.
\end{theorem}

We now introduce the following concept appearing in the definition of the class $\Cc$:
\begin{definition}\label{def:co-hopf}
A group is said to be co-Hopfian if it is not isomorphic to any of its proper subgroups. 
\end{definition}
 A straightforward characterization of the co-Hopfian property is that of  a group $G$ for whom every injective endomorphism is an automorphism.
It is clear that every finite group is co-Hopfian, though the class of co-Hopfian groups is larger, as the reader can check at the list of examples. We finally introduce our class  of groups:
\begin{definition}\label{def:C}
We say that a group $G$ is in the class $\Cc$ if the following three conditions hold:
\begin{itemize}
\item[i)] There exists a connected locally finite graph $\G$ such that $G\cong\Aut(\G)$.
\item[ii)] $G$ is co-Hopfian.
\item [iii)] Every abelian normal subgroup of $G$ is torsion.
\end{itemize}
\end{definition}

\begin{example}[Finite groups are in $\Cc$]\label{ex:finite}
Given a finite group $G$, we have:
\begin{itemize}
\item[i)] $G$ is countable, and according to Theorem \ref{th:groups&graphs}, there exists a graph $\G$ such that $G\cong\Aut(\G)$.
\item[ii)] $G$ is co-Hopfian since the only subgroup of $G$ of order $|G|$ is $G$.
\item[iii)] $G$ is torsion, and therefore any abelian normal subgroup of $G$ is torsion.
\end{itemize}
\end{example}

\begin{example}[Tarski groups are in $\Cc$]\label{ex:Tarski}
Recall that an infinite group $G$ is a Tarski group for the prime $p$, if  every non-trivial proper subgroup of $G$ is of order $p$.  Tarski groups are known to exist for big enough primes \cite{ol1, ol2}. Then:
\begin{itemize}
\item[i)] $G$ is finitely generated ($G$ is generated by any two non commuting elements), and therefore countable. By Theorem \ref{th:groups&graphs}, there exists a connected locally finite graph $\G$ such that $G\cong\Aut(\G)$.
\item[ii)] $G$ is co-Hopfian since proper subgroups of $G$ are finite, while $G$ is not.
\item[iii)] $G$ is simple, and therefore there is no abelian normal proper subgroup of $G$.
\end{itemize}
\end{example}
Previous examples have in common that they are finitely generated. This is not a necessary condition as  we now illustrate.

\begin{example}[Solvable Artinian groups are in $\Cc$]{\label{ex:artin}}
Recall that a group $G$ is an Artinian group if it satisfies the minimal condition, that is, if every strictly descending chain of subgroups of $G$, $G_1>  G_2 > G_3 > \cdots$, is finite.
Then:
\begin{itemize}
\item[i)]  $G$ is a countable (see \cite[p.\ 192]{Kurosh}).  Hence, by Theorem \ref{th:groups&graphs}, there exists a connected locally finite graph $\G$ such that $G\cong\Aut(\G)$.
\item[ii)] The co-Hopfian property is immediatly obtained from the minimal condition.
\item[iii)] $G$ is torsion \cite[Theorem 1.8.2]{Hall}. Hence, every abelian normal subgroup is torsion too.
\end{itemize}
An example in this class is the Pr\"ufer $p$-group, $p$ a prime, defined to be the Sylow $p$-group of $\mathbb Q / \mathbb Z$, that is the set of elements
of $\mathbb Q/ \mathbb Z$ whose order is a power of p.  Pr\"ufer groups are not finitely generated, in contrast with the last two examples who also were Artinian countable groups. Notice that the class of solvable Artinian groups is larger, since  it contains all the extensions of abelian Artinian groups by finite solvable groups.
\end{example}
\begin{example}[Mapping class groups of some compact surfaces are in $\Cc$]\label{ex:mapclass}
Let $S$ be a connected compact orientable surface. Suppose that $S$ is not a sphere with at most four holes,  nor a torus with at most two holes. Let $G$ be the mapping class group of $S$ \cite{Ivanov}. Then:
\begin{itemize}
\item[i)] $G$ is finitely generated \cite{Dehn2}, and therefore countable. Hence, by Theorem \ref{th:groups&graphs}, there exists a connected locally finite graph $\G$ such that $G\cong\Aut(\G)$.
    
\item[ii)] $G$ is co-Hopfian according to \cite[Theorem 1]{coHopf-Teichm} (if $S$ has positive genus) and \cite[Theorem 1.2]{coHpof-Braid} (if $S$ is a punctured sphere).
    
\item[iii)] $G$ does not contain any non trivial torsion-free abelian normal subgroup. Indeed, if $H\leq G$ is a torsion-free abelian normal subgroup, then $H$ contains only reducible or pseudo-Anosov classes \cite[Section 13.2]{Farb-Margalit}. Assume there exists a pseudo-Anosov class $x\in H$, and let $\F=(\F^u,\F^s)$ be a pair of transverse measured foliations associated to (preserved by) $x$ \cite[Section 13.2.3]{Farb-Margalit}. Then $H\leq C_G(x)\leq\{y\in G : y(\F)=\F\}$ \cite[Lemma 1.(3)]{McCarthy}. Since $H$ is normal, $yxy^{-1}\in H$ for every $y\in G$, that is, $yxy^{-1}(\F)=\F$. But we also have that $yxy^{-1}(y(\F))=y(\F)$, and therefore $\F=y(\F)$ (up to a Whitehead move), thus $y\in H$ what would imply that $G = H$ is abelian! Therefore $H$ cannot contain any pseudo-Anosov class. Assume now there exists a reducible class $x\in H$, and let $C=\{c_i\}$ be its canonical reduction system \cite[Section 13.2.2]{Farb-Margalit}. Now given $y\in G$, $yxy^{-1}$ is another reducible class whose canonical reduction system is $y(C)$. Since $H$ is normal, $yxy^{-1}\in H$ thus $x$ and $yxy^{-1}$ commute and a straightforward calculation shows that $x$ preserves de reduction system $y(C)$. But $y$ can be chosen such that $y(C)$ intersects non trivilly with $C$, what is a contradiction with the fact that $C$ is a canonical reduction system. Hence $H$ cannot contain any reducible element, and $H$ must be trivial. 
\end{itemize}
\end{example}

\section{From automorphisms of graphs to automorphisms of DGA's}
This section relies heavily on previous results of the authors  in \cite[Section 2]{CV2}. In that previous work, for a given finite graph $\G$, we construct a finite-type DGA, $(A,d)$, whose automorphisms group is closely related to the automorphism group of the graph. In the present work, for our purpose, we do not need the finite-type condition, therefore we extend our construction to locally finite graphs. Following notation in \cite{FHT}, we prove the following theorem.
 \begin{theorem}\label{2}
Let $\G=(V,E)$ be a connected locally finite graph, then there exists a DGA $(A,d)$ such that $$\Aut(A,d) \cong K\rtimes\Aut(\G)$$ where $K$ is an abelian torsion-free group.
\end{theorem}
\begin{proof}
Let $(A,d)=\big(\Lambda(x_1,x_{2},y_1, y_2, y_3, z)\otimes\Lambda(x_v, z_v\vert
v\in V),d\big)$ be the free commutative 
graded algebra on $(2 \vert V \vert + 6)$ generators in dimensions:
$$\vert x_1\vert = 8,\; 
\vert x_2\vert = 10,\; 
\vert y_1\vert = 33,\; 
\vert y_2\vert = 35,\; 
\vert y_3\vert = 37,\; 
\vert z\vert = 119,\; 
\vert x_v\vert = 40,\; 
\vert z_v\vert = 119,$$
and  differentials:
\begin{equation*}
\begin{array}{ll}
d(x_1)=&0\\
d(x_2)=&0\\
 d(y_1)=&x_1^3x_2\\
d(y_2)=&x_1^2x_2^2\\
d(y_3)=&x_1x_2^3\\
d(x_v)=&0\\
d(z) =&y_1y_2x_1^4x_2^2-y_1y_3x_1^5x_2+y_2y_3x_1^6+x_1^{15}+x_2^{12}\\
d(z_v)=&x_v^3+\sum_{[v,w]\in E}x_vx_wx_2^4.
\end{array}
\end{equation*}
Observe that $(A,d)$ is well defined since $\G$ is locally finite and therefore the sum defining the differential of  elements $z_v$ is, indeed,  finite. According to \cite[Section 2]{CV2}, given $f \in\Aut(A,d)$, there exists  $\sigma \in \Aut (\G)$  such that:
\begin{equation*}
\begin{array}{rl}
f(x_1) =& x_1\\
f(x_2) =& x_2\\
f(y_1) =& y_1\\
f(y_2) =& y_2\\
f(y_3) =& y_3\\
f(x_v)=&  x_{\sigma(v)} \\
f(z) =& z + d(m_z)\\
f(z_v)=& z_{\sigma(v)}  +d(m_{z_v})
\end{array}
\end{equation*}
with  $\vert m_z \vert = \vert m_{z_v} \vert = 118$ elements in $A$.  Therefore, the projection over the module of indecomposable elements of $A$ provides an split exact sequence of groups
$$K\longrightarrow\Aut(A,d)\longrightarrow\Aut(\G)$$
where $f \in K$ if and only if
\begin{equation*}
\begin{array}{rl}
f(x_1) =& x_1\\
f(x_2) =& x_2\\
f(y_1) =& y_1\\
f(y_2) =& y_2\\
f(y_3) =& y_3\\
f(x_v)=&  x_{v} \\
f(z) = & z + d(m_z)\\
f(z_v)= & z_{v} + d(m_{z_v})
\end{array}
\end{equation*}
with $\vert m_z \vert = \vert m_{z_v} \vert = 118$ elements in
$A$. Hence, by degree reasons,  $f(m_z)=m_z$ and $f(m_{z_v})=m_{z_v}$.
These two equalities force the following  ones: $f^n(z) = z + nd(m_z)$ and $f^n(z_v) = z_v + nd(m_{z_v})$,  for any integer  $n$. Hence,  $f^n\ne 1$ for $n\ne 0$ and we can conclude that $K$ is torsion-free. 
Now, to prove that $K$ is abelian, a direct calculation shows that $[f,g] =1$.

\end{proof}

\section{Proof of Theorem \ref{main}}\label{sec:main}
One implication is immediate. Assume now that for every DGA, $(A,d)$, $G\leq\Aut(A,d)$ if and
only if $H\leq\operatorname{Aut}(A,d)$. Since $G$ and $H$ are in $\Cc$, there exist connected locally finite graphs $\G_G$ and $\G_H$ such that $G\cong \Aut(\G_G)$ and $G\cong \Aut(\G_G)$. Therefore, according to Theorem \ref{2}, there exists DGA's $(A_G,d)$ and $(A_H,d)$ such that $\Aut(A_G,d)\cong K_1\rtimes G$ and $\Aut(A_H,d)\cong K_2\rtimes H$, where $K_1$ 
and $K_2$ are abelian torsion-free groups.

Now we use that, since $G\leq\Aut(A_G,d)$, by hypothesis $H\leq\Aut(A_G,d)$.  As $K_1$ is an abelian normal subgroup of $\Aut(A_G,d)$, then $H\cap K_1$  is an abelian normal subgroup of $H$, and therefore it must be torsion because $H\in\Cc$.  But $K_1$ is torsion-free, hence $H\cap K_1$ must be trivial. Thus, the inclusion $H\leq K_1\rtimes G$ leads to an inclusion $H\leq G$. A similar argument using that $H\leq\Aut(A_H,d)$ leads to an inclusion $G\leq H$. 

Therefore, we have a series of inclusions $G\leq H\leq G$, and since $G$ is co-Hopfian, we have that $G \cong H$, what concludes the proof.

\begin{remark}
Notice that the proof of Theorem \ref{main} only uses the co-Hopfian property of one of the groups involved. Therefore, it would be possible to enlarge the class $\Cc$ by dropping the co-Hopfian requirement. It should be then added as an hypothesis in the theorem, for at least one of the groups.
\end{remark}

\noindent{{\bf{Acknowledgements.}} The authors want to thank Rosa Mar\'ia Fern\'andez-Rodr\'iguez and  Juan Gonz\'alez-Meneses for
fruitful conversations and suggestions.

\end{document}